\documentclass[12pt]{amsart}
\usepackage{amsmath,amsfonts,latexsym,graphicx,amssymb,url}
\usepackage{amsmath,txfonts,pifont,bbding,pxfonts,manfnt}
\usepackage{psfrag}
\usepackage{wrapfig, subfigure,graphicx}
\usepackage{hyperref}
\usepackage[active]{srcltx}
\usepackage{pdfsync}
\usepackage{amssymb}
\usepackage{amsthm}
\usepackage{amsmath}
\usepackage{graphicx,cases,datetime}
\usepackage{color}
\usepackage{soul, xcolor}
\usepackage{mathtools}
\usepackage{enumerate}
\usepackage{enumitem}

\usepackage[symbol]{footmisc}

\allowdisplaybreaks
\numberwithin{equation}{section}
\newcommand{\dist}{\text{dist}} 

\newcommand{\R}{{\mathbb R}} 

\def \S {{\mathcal S}}

\def \e {\varepsilon}

\def \v{\mathcal O}

\def \e {\mathbf e}
\def\t{\mathcal T}

\def\d{\overline D}
\def\m{\overline{\Omega}}

\def\b{\textbf{b}}

\renewcommand{\(}{\left(}
\renewcommand{\)}{\right)}

\newtheorem{theorem}{Theorem}[section]
\newtheorem{corollary}[theorem]{Corollary}
\newtheorem{lemma}[theorem]{Lemma}
\newtheorem{proposition}[theorem]{Proposition}
\newtheorem{definition}[theorem]{Definition}
\newtheorem{remark}[theorem]{Remark}

\begin{document}

\title[Near Field Refractor Problem, \today]{On the Numerical Solution of the \\Near Field Refractor Problem}
\author[C. E. Guti\'errez and H. Mawi]{Cristian E. Guti\'errez and Henok Mawi}
\thanks{\today. \\The first author was partially supported by NSF grant DMS--1600578, and the second author was partially supported by NSF grant HRD--1700236.}
\address{Department of Mathematics\\Temple University\\Philadelphia, PA 19122}
\email{gutierre@temple.edu}
\address{Department of Mathematics, Howard University, Washington, D.C. 20059}
\email{henok.mawi@howard.edu}

\begin{abstract}
A numerical scheme is presented to solve the one source near field refractor problem to arbitrary precision and it is proved that the scheme terminates in a finite number of iterations.
The convergence of the algorithm depends upon proving appropriate Lipschitz estimates for the refractor measure. 
The algorithm is presented in general terms and has independent interest.
\end{abstract}

\maketitle

\tableofcontents

\section{Introduction}
Let $\Omega \subset S^{n-1}$ be a domain and suppose for each point $x \in \Omega,$ a light ray with direction $x$ emanates from a punctual source at the origin $O,$ with intensity density function $f(x),$ where $f>0 \, a.e.$ on $\Omega$  and $f \in L^1(\Omega)$ . Suppose $D \subset \R^n$, the target we want to illuminate, is a domain contained in an $n-1$ dimensional hyper-surface,  with $\overline D$  compact and $O \notin \overline D.$ Let $\mu$ be a Radon measure on $\overline D$ satisfying the mass balance condition 
\[
\mu (\overline D) = \int_{\Omega} f(x) \, dx.
\]
Given two homogeneous and isotropic media $I$ and $II$ with refractive indices $n_1,n_2$, respectively, so that the point source at $O$ is surrounded by medium $I$ and the target domain $D$ is surrounded by medium $II$,
the {\emph {near field refractor problem}  is to find an interface $\mathcal S$ between media $I$ and $II$ parametrized by $\mathcal S = \{\rho(x) x : x \in \overline \Omega\}$ so that each ray with direction $x\in \Omega$ is refracted into $D$ according to Snell law and so that the energy conservation condition
\[
\int_{\mathcal T_{\mathcal S} (F)} f(x) \, dx = \mu(F)
\]
holds for all $F\subset D$, where $\mathcal T_{\mathcal S} (F)$ represents the directions $x\in \Omega$ that are refracted into $F$, see Definition \ref{def:definition of near field refractor}.
Existence of solutions to this problem is obtained in \cite{GH14}.

The purpose of this paper is to present an iterative scheme to find approximate solutions for this problem with arbitrary precision when $\mu$ is discrete measure and prove that the scheme gives the desired result in finite number of iterations.
The physical problem is three dimensional, but we carry out the analysis in $n$ dimensions.

A similar iterative scheme was developed in \cite{DGM17} to solve the far field refractor problem, extended in \cite{AG17} to deal with generated Jacobian equations, and in particular, used in \cite{K14} for mass transport problems with cost functions satisfying the MTW condition given in \cite{MTW05}.
The algorithm originates in pioneering works by Caffarelli, Kochengin and Oliker \cite{CKO99} for reflectors and by Bertsekas \cite{B79} for the assignment problem.
However, a major advance and simplification to solve these kind of problems numerically is introduced in  \cite{DGM17}  and  \cite{AG17} by showing that an appropriate mapping satisfies a Lipschitz condition. 
This essential step guarantees that the algorithm converges in a finite number of iterations, and
we stress that this does not require smoothness of the density function $f$.

The difficulty in extending these ideas when dealing with the near field refractor problem is that it has a complicated geometrical structure given by Descartes ovals requiring non trivial analytical estimates for the derivatives of these ovals, and it does not have a mass transport structure. 
Moreover, we present an abstract form of the algorithm having independent interest that might be useful to solve other problems with similar features.

The plan of the paper is as follows. Section \ref{sec:set up and definitions} contains preliminary results, the set up, and definitions needed in the rest of the paper.
In Section \ref{sec:estimates of derivatives of ovals} we present estimates of the derivatives of  Descartes ovals and lower gradient estimates under structural conditions on the discretization of the target $D$, Proposition \ref{prop:structural condition for lower estimates of gradients}. 
We will use these estimates in Section \ref{sec:Lipchitz estimate of the refractor map} to prove a one sided Lipschitz estimate of the refractor measure, Theorem \ref{thm:one sided Lipschitz estimate}. 
A Lipschitz estimate is a crucial  ingredient to prove that the abstract algorithm introduced in Section \ref{sec:abstract algorithm}, terminates in finitely many steps as shown in Section \ref{subsec:algorithm converges in a finite number of steps}; in particular, when applied to the near field refractor problem. In Section \ref{subsec:geometric assumptions on Omega and D}, we introduce a class of admissible vectors that will be used to apply the abstract algorithm to the near field refractor.   
Finally, in Section \ref{sec:final solution of the refractor problem} we show the application of the algorithm to solve the near field refractor problem.

\section{Set up and Definitions}\label{sec:set up and definitions}

In this section, we shall recall Snell's law of refraction, discuss some properties of the building blocks from which we construct near field refractors, and 
state geometric conditions between the set of incident directions and target to guarantee existence of solutions. In addition, we give the precise statement of the problem solved in the paper. 
\subsection{Snell's Law of refraction}
If from a point source of light located at the origin and surrounded by media $I$, with refractive index $n_1$, a ray of light emanates with unit direction $x$ and strikes an interface $\mathcal S$ between medium I and medium II  at a point $P$, then this light ray is refracted into the unit direction $m$ in medium $II$, with refractive index $n_2$, according to Snell's law given in vector form as 
\begin{equation}\label{vectorformofsnellslawwithlambda}
x - \kappa m = \lambda \nu
\end{equation}
with $\kappa=n_2/n_1$ and $\lambda= x \cdot \nu-\kappa \sqrt{1-\kappa^{-2}(1-(x \cdot \nu)^{2})}$, where $\nu$ is the unit normal at $P$ pointing towards medium $II$.
From this the standard Snell's law follows: $n_1\sin \theta_i=n_2\sin \theta_r$, with $\theta_i$ the angle of incidence between $x$ and $\nu$, and $\theta_r$ the angle of refraction between $m$ and $\nu$.
When $\kappa<1$, depending on the angle of incidence total internal reflection may occur, i.e., incident light may be totally reflected back into medium $I$ and not transmitted to medium II. If $x\cdot \nu\geq \sqrt{1-\kappa^2}$ or equivalently, $x\cdot m\geq \kappa$, then there is no total internal reflection;
see \cite[Sect. 2.1]{GH09}.

We will assume throughout the paper that $\kappa<1$. The analysis for $\kappa >1$ is similar.

\subsection{Descartes Ovals}
The treatment of the near field refractor problem requires the use of Descartes ovals, which have a special refraction property.
For $P \in \R^n$ and $\kappa |P| < b < |P|,$ a refracting  Descartes oval is the surface
\[
\v_b(P)  = \{h(x,P,b)x : x \in S^{n-1}, \, x\cdot P \geq b \}
\]
where
\begin{equation}\label{eq:definition of h(x,P,b)}
h(x,P,b) = \dfrac{(b-\kappa^2 x \cdot P)- \sqrt{(b-\kappa^2 x \cdot P)^2 - (1-\kappa^2) (b^2 - \kappa^2 |P|^2)}}{1 - \kappa^2}.
\end{equation}
If the region inside the oval $\{h(x,P,b)x : x \in S^{n-1} \}$ is made of a material with refractive index $n_1$ and the outside made of material with refractive index $n_2$, 
then using Snell's law it can be shown that each light ray emanating from the origin $O$ and having direction $x\in S^{n-1}$ with $x\cdot P\geq b$ is refracted by the oval $\v_b(P)$ into the point $P.$  See \cite[Sect. 4]{GH14} for detailed discussion.
\begin{figure}[htp]
  \centering
    \includegraphics[width=1\textwidth]{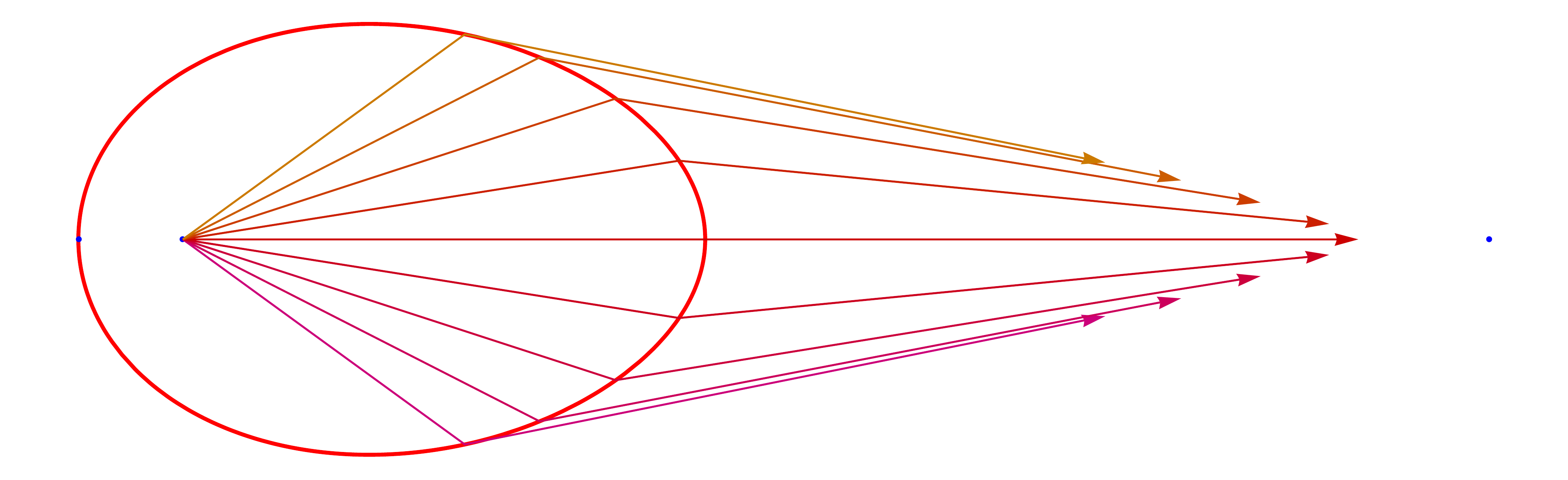}
    \caption{Refracting oval} 
    \label{pic:refracting oval}   
\end{figure}

 \subsection{Statement of the Problem}\label{subsec:statement of the problem} As in  \cite{GH14} we will impose the following two geometric configuration conditions on $\Omega$ and $D$ to formulate the main problem. The first condition is to avoid total internal reflection and the second is to guarantee that the target doesn't block itself:
\begin{enumerate}[label=\textbf{H.\arabic*}]
\item there exists $\tau$, with $0<\tau<1-\kappa$, such that $x\cdot P\geq (\kappa+\tau)|P|$ for all $x\in \Omega$ and $P\in D$;\label{eq:hypotheses 1}
\item let $0<r_0<\dfrac{\tau}{1+\kappa}\,\dist(O,D)$ and $Q_{r_0}=\{t\,x:x\in \Omega, 0\leq t<r_0\}$. Then given $X\in Q_{r_0}$ each ray emanating from $X$ intersects $D$ in at most one point. \label{eq:hypotheses 2} 
\end{enumerate}
Note that if $r_0$ satisfies \ref{eq:hypotheses 2}, then $r_0<\dfrac{1-\kappa}{1+\kappa}|P|$ for $P\in D$.
We shall prove that if   
\begin{equation}\label{eq:lower and upper estimate of b to avoid internal reflection}
\kappa|P|<b\leq (1+\kappa) \,r_0+\kappa |P|
\end{equation}
then
the oval $\v_b(P)$ refracts all directions $x\in \Omega$ into $P$. For this, we only need to verify that there is no internal reflection, that is, $x\cdot P\geq b$ for all $x\in \Omega$.
Indeed, 
\begin{align}\label{eq:b less than or equal to x.P}
b&\leq (1+\kappa) \,r_0+\kappa |P|
<(1+\kappa)\,\dfrac{\tau}{1+\kappa}\,\dist(0,D)+\kappa |P|\quad \text{from \ref{eq:hypotheses 2}}\notag\\
&\leq
(\tau+\kappa)\,|P|\leq x\cdot P \quad \text{from \ref{eq:hypotheses 1}.}
\end{align}

Near field refractors are then defined as follows. 
\begin{definition}\label{def:of refractor}
A surface $\mathcal S = \{x\rho(x): x \in \m\} \subset Q_{r_o}$ is  said to be a near field  refractor if for any point $y\rho(y) \in \mathcal S$ there exist point $P \in \d$ and $\kappa |P|<b<|P|$ such that the refracting oval $\mathcal \v_b(P)$ supports $\mathcal S$ at $y \rho(y)$, i.e. $\rho(x) \leq h(x,P,b)$ for all $x \in \m$ with equality at $x = y.$
\end{definition}
We remark that if \eqref{eq:hypotheses 1} and \eqref{eq:hypotheses 2} hold, then each near field refractor is Lipschitz \cite[Lemma 5.3]{GH14}.

The refractor map is as follows.
\begin{definition}\label{def:definition of near field refractor}
Given a near field refractor $\mathcal S,$ the near field refractor mapping of $\mathcal S$ is the multi-valued map defined for $x \in \m$ by
\[
\mathcal R_{\mathcal S}(x) = \{P \in \d: \textrm { there exists a supporting oval} \,  \v_b(P) \, \textrm{to} \, \mathcal S  \, \textrm{at} \,  \rho(x)x\}.
\]
Given $P \in \d$ the near field tracing mapping of $\S$ is defined by
\[
\mathcal T_{\S}(P) = \mathcal R^{-1}_{\S}(P) = \{x \in \m: P \in \mathcal R_{\S} (x)\}.
\]
\end{definition}
Suppose that we are given a nonnegative $f \in L^1(\m)$, $f(x)$ represents the intensity of the light ray emanating from $O$ with direction $x$.  We recall the definition of near field refractor measure, see \cite[Formula (5.5)]{GH14}. 
\begin{definition}\label{def:definition of refractor measure}
The near field refractor measure associated with the near field refractor $\S$ and the function $f \in L^1(\m)$ is the finite Borel measure given by  
\[
\mathcal M_{\S,f} (F) = \int_{\mathcal T_{\S}(F)} f \, dx
\]
for every Borel set $F\subset \d.$
\end{definition}

Given a Radon measure $\mu$ defined on $D$ and the energy conservation condition $\int_{\Omega} f\,dx = \mu(D),$ {\it the near field refractor problem} is to find a near field refractor $\S$ such that 
\[
\mathcal M_{\S,f} (F) = \mu(F)
\]
for any Borel set $F \subset \d.$
Under conditions \ref{eq:hypotheses 1} and \ref{eq:hypotheses 2} above, the existence of such a surface $\S$ is proved in \cite{GH14}.
In particular, given distinct points $P_1, \cdots, P_N,$  in $\d,$ positive numbers $g_1, \cdots, g_N,$ a nonnegative function $f \in L^1(\Omega)$ such that 
\begin{equation}\label{energy conservation condition}
\int_{\m} f(x)dx = \sum_{i=1}^N g_i 
\end{equation}
and $\kappa |P_1| < b_1< \kappa |P_1| + r_0 \dfrac{(1-\kappa)^2}{1+\kappa},$ ($r_0$ in Hypothesis \ref{eq:hypotheses 2} above) there exists a unique vector $\b = (b_1, b_2, \cdots, b_N) \in \prod_{i=1}^N\( \kappa |P_i|,|P_i|\)$ such that the poly-oval 
\begin{equation}\label{eq:definition of poly-oval}
\S(\b) = \{\rho(x)x : x \in \m \, \textrm{and} \,\rho(x) = \min_{1 \leq i \leq N} h(x, P_i, b_i)\}
\end{equation}
is a near field refractor satisfying
\[
\mathcal M_{\S(\b), f} ({P_i}) = g_i;
\]
see \cite[Thm. 5.7]{GH14}.

{\it The main purpose of this paper is} to discuss an iterative scheme to construct this refractor with arbitrary precision and to show that the scheme converges in a finite number of steps.  That is, given $f; g_1,\cdots ,g_N;P_1,\cdots ,P_N, b_1$ satisfying $\kappa |P_1| < b_1 \leq \kappa |P_1|+r_0 \dfrac{(1-\kappa)^2}{1+\kappa}$ and  $\epsilon > 0$ we seek a vector $\b = (b_1, b_2, \cdots, b_N) \in \prod_{i=1}^N\( \kappa |P_i|,|P_i|\)$, which depends on $\epsilon$, such that the poly-oval refractor $\S(\b)$ satisfies 
\begin{equation}\label{conditionfornumericalsolution}
|\mathcal M_{\S(\b),f}(P_i) - g_i| \leq \epsilon,\qquad 1\leq i\leq N.
\end{equation}

\subsection{Properties of the refractor mapping}
In this subsection we will prove some results that will be needed to apply the algorithm from Section \ref{sec:abstract algorithm} to solve the main problem. 
In the proof of these results, and in the subsequent sections, the following part of   \cite[Lemma 4.1]{GH14}  will be used frequently.

\begin{lemma}\label{lm:minandmaxofovalskappa<1}
Let $0<\kappa <1$, $h(x,P,b)$ given by \eqref{eq:definition of h(x,P,b)}, and assume that $\kappa |P|<b<|P|$.
Then 
\begin{equation}\label{eq:lowerestimateofrho}
\min_{x\in S^{n-1}}h(x,P,b)
=\dfrac{b-\kappa |P|}{1+\kappa},
\qquad 
\text{and}\qquad \max_{x\in S^{n-1}}h(x,P,b) =\dfrac{b-\kappa |P|}{1-\kappa}.
\end{equation}
\end{lemma}

We begin with the following monotonicity property.
\begin{lemma}\label{monotonicityoftrace}
Let $\b = (b_1, \ldots, b_N)$ and $\b^* = (b^*_1, \ldots, b_N^*)$ be in $\prod_{i=1}^N\( \kappa |P_i|,|P_i|\)$ 
Suppose that for some $l$, $b^*_l \leq b_l$ and for all $i \neq l,$ $b_i^* = b_i$, where $1 \leq l, i\leq N.$ Then
\begin{equation}\label{eq:inclusionwithell}
\t_{\S(\b)}(P_l) \subseteq \t_{\S(\b^*)}(P_l)
\end{equation}
and 
\begin{equation}\label{eq:inclusionwithinotell}
\t_{\S(\b^*)}(P_i) \subseteq \t_{\S(\b)}(P_i) \  \textrm{for} \ i \neq l,
\end{equation}
where the inclusions are up to a set of measure zero.
Consequently
\[
\mathcal M_{\S(\b),f}(P_l) \leq \mathcal M_{\S(\b^*),f}(P_l) \ \textrm{and} \ \mathcal M_{\S(\b),f}(P_i) \geq \mathcal M_{\S(\b^*),f}(P_i) \  \textrm{for} \ i \neq l.
\]
\end{lemma}
\begin{proof}
From its definition, $\S(\b)$ is differentiable a.e., so the set of singular points has surface measure zero.
We use here that if $x_0\in \t_{\S(\b)}(P_l)$ is not a singular point, then the oval $\mathcal O_{b_l}(P_l)$ supports $\S(\b)$ at $x_0$; this holds for any near field refractor $\S(\b)$ and any $1\leq l\leq N$; see \cite[Proof of Lemma 5.4]{GH14}.

We first prove \eqref{eq:inclusionwithinotell} when $x_0$ is not a singular point of $\S(\b^*)$. Since $b_l^*\leq b_l$, we obviously have $\rho^*(x)\leq \rho(x)$ for all $x\in \Omega$, where $\rho^*$ is the parametrization of $\S(\b^*)$ and $\rho$ is the parametrization of $\S(\b)$. 
Suppose $i\neq l$ and let 
$x_0\in \t_{\S(\b^*)}(P_i)$. Then, since $x_0$ is not a singular point of $\S(\b^*)$, the oval with polar radius $h(x,P_i,b_i)$ supports $\S(\b^*)$ at $x_0$. We have $\rho(x)\leq h(x,P_i,b_i)$. Therefore 
\[
h(x_0,P_i,b_i)=\rho^*(x_0)\leq \rho(x_0)\leq h(x_0,P_i,b_i),
\]
that is, $x_0\in \t_{\S(\b)}(P_i)$. 
\newline
We now prove \eqref{eq:inclusionwithell}. That is, we prove that if $x_0$ is neither a singular point of $\S(\b)$ nor a singular point of $\S(\b^*)$, and $x_0\in \t_{\S(\b)}(P_l)$, then $x_0\in \t_{\S(\b^*)}(P_l)$. We may assume $b_l^*<b_l$.
We have that the oval $\mathcal O_{b_l}(P_l)$ supports $\S(\b)$ at $x_0$. We claim that the oval with polar radius $h\(x,P_l,b_l^*\)$ supports $\S(\b^*)$ at $x_0$. 
Suppose this is not true. Since by definition $\rho^*(x)\leq h\(x,P_l,b_l^*\)$, we would have $\rho^*(x_0)< h\(x_0,P_l,b_l^*\)$.
So $\rho^*(x_0)=h\(x_0,P_j,b_j\)$  for some $j\neq l$, and therefore $h\(x,P_j,b_j\)$  supports $\S(\b^*)$ at $x_0$.
Since $x_0$ is not a singular point of $\S(\b^*)$, then by the inclusion previously proved we get that $x_0\in \t_{\S(\b)}(P_j)$. Since $j\neq l$ we obtain that $x_0$ is a singular point of $\S(\b)$, a contradiction.
\end{proof}

\begin{lemma}\label{forlarge_b_energyis0}
Let $\b = (b_1, \ldots, b_N) \in \prod_{i=1}^N\( \kappa |P_i|,|P_i|\)$. Consider the family of refractors obtained from $\S(\b)= \{ \rho(x)x : x \in \Omega\},$ by changing only $b_i$ and fixing $b_j$ for all $j \neq i.$ Then
$\mathcal M_{\S(\b),f}(P_i)\to \int_\Omega f(x)\,dx$ as $b_i\to \(\kappa |P_i|\)^+$.
\end{lemma}
\begin{proof}
We have $h(x,P_i,b_i)\leq \dfrac{b_i-\kappa |P_i|}{1-\kappa}$ and also
$h(x,P_j,b_j)\geq \dfrac{b_j-\kappa|P_j|}{1+\kappa}$  for all $x\in \Omega$, 
from Lemma \ref{lm:minandmaxofovalskappa<1}. Let $\delta=\min_{j\neq i}\dfrac{b_j-\kappa|P_j|}{1+\kappa}$. For $b_i$ with 
$\dfrac{b_i-\kappa |P_i|}{1-\kappa}<\delta$, we then get $h(x,P_i,b_i)<h(x,P_i,b_i)$ for all $x\in \Omega$ and $j\neq i$. So $\rho(x)=h(x,P_i,b_i)$ and so $\t_{\S(\b)}(P_i)=\Omega$ and (a) follows.

\end{proof}

\section{Estimates for derivatives of the polar radii of ovals}\label{sec:estimates of derivatives of ovals}
 In this section, we will obtain bounds for the derivatives of the polar radius $h(x, P, b).$ In particular, we will prove the lower gradient estimate \ref{eq:lower bound for mathcal Gij} which will be used in the next section to prove a Lipschitz property of the refractor measure.
From \eqref{eq:definition of h(x,P,b)} we have 
\[
h(x,P,b) = \dfrac{(b-\kappa^2 x \cdot P)- \sqrt{\Delta(x\cdot P,b,|P|)}}{1 - \kappa^2},
\]
where 
\[
\Delta(t,b,|P|):=\(b-\kappa^{2}\, t\)^{2}-\(1-\kappa^{2}\)\(b^{2}-\kappa^{2}|P|^{2}\).
\]
By calculation 
\begin{equation}\label{eq:definition of Deltat}
\Delta(t,b,|P|) = \kappa^{2}\((b- t)^{2}
+(1-\kappa^{2})\(|P|^{2}-t^{2}\)\).
\end{equation}
As a function of $t$, $\Delta$ is increasing in the interval $(b/\kappa^2,+\infty)$ and decreasing in the interval $(-\infty,b/\kappa^2)$.
Let $t=x\cdot P$ with $|x|\leq 1+\epsilon$, so $-(1+\epsilon)|P|\leq t\leq (1+\epsilon)|P|$, and we have
\[
\min_{|x|\leq 1+\epsilon}\Delta(x\cdot P,b,|P|)=
\min_{-(1+\epsilon)|P|\leq t\leq (1+\epsilon)|P|}\Delta(t,b,|P|).
\]
Let $\epsilon>0$ be such that $1+\epsilon<1/\kappa$. Since $b>\kappa\,|P|$, it follows that $\left[-(1+\epsilon)|P|, (1+\epsilon)|P|\right]
\subset \(-\infty,b/\kappa^2\)$. Therefore 
\[
\min_{-(1+\epsilon)|P|\leq t \leq (1+\epsilon)|P|}\Delta(t,b,|P|)=\Delta\((1+\epsilon)|P|,b,|P|\),
\]
and 
\[
\max_{-(1+\epsilon)|P|\leq t \leq (1+\epsilon)|P|}\Delta(t,b,|P|)=\Delta\(-(1+\epsilon)|P|,b,|P|\).
\]
We have 
\begin{align*}
\Delta\(-(1+\epsilon)|P|,b,|P|\)
&=
\kappa^2\,\(\((1+\epsilon)|P|+b\)^2+(1-\kappa^2)|P|^2\(1-(1+\epsilon)^2\)\)\\
&\leq
C(\kappa)\,|P|^2
\end{align*}
for $\kappa|P|<b<|P|$, so
\begin{equation}\label{eq:upper bound for Delta}
\Delta(x\cdot P,b,|P|)\leq C(\kappa)\,|P|^2,\qquad \text{for $|x|\leq 1+\epsilon$.}
\end{equation}
Let us estimate $\Delta\((1+\epsilon)|P|,b,|P|\)$ from below.
For this, we assume $b$ satisfies \eqref{eq:lower and upper estimate of b to avoid internal reflection} and recall assumptions \ref{eq:hypotheses 1} and \ref{eq:hypotheses 2}.

We write
\begin{align*}
\Delta((1+\epsilon) |P|,b,|P|)
&=\kappa^2\,\(\((1+\epsilon)|P|-b\)^2+(1-\kappa^2)|P|^2\(1-(1+\epsilon)^2\)\)\\
&=\kappa^2\,
\(
\(|P|-b\)^2
+
2\,\epsilon \,|P|\,(|P|-b)
+\epsilon\,\(\epsilon-(1-\kappa^2)\,\(2+\epsilon\)\)\,|P|^2\)\\
&\geq 
\kappa^2\,
\(
\(|P|-b\)^2
+\epsilon\,\(\epsilon-(1-\kappa^2)\,\(2+\epsilon\)\)\,|P|^2\).
\end{align*} 
We have 
\begin{align*}
|P|-b&>(1-\kappa)|P|-(1+\kappa)\,r_0\qquad \text{from \eqref{eq:lower and upper estimate of b to avoid internal reflection}}\\
&\geq \dfrac{1+\kappa}{\tau}(1-\kappa)r_0-(1+\kappa)r_0\qquad \text{from \ref{eq:hypotheses 2} since $P\in D$}\\
&=(1+\kappa)\(\dfrac{1-\kappa}{\tau}-1\)\,r_0:=\delta>0
\end{align*}
since $\tau<1-\kappa$ from \ref{eq:hypotheses 1}.
Since $D$ is a bounded set, $|P|\leq M$ for all $P\in D$, and since $\epsilon-(1-\kappa^2)\,\(2+\epsilon\)<0$ for $\epsilon>0$ small, we get 
\[
\epsilon\,\(\epsilon-(1-\kappa^2)\,\(2+\epsilon\)\)\,|P|^2
\geq 
\epsilon\,\(\epsilon-(1-\kappa^2)\,\(2+\epsilon\)\)\,M^2
\geq -\delta^2/2
\]
choosing $\epsilon>0$ small enough.

Therefore we obtain that there are structural constants $C_0>0$ and $\epsilon>0$ such that
\begin{equation}\label{eq:lower bound for quantity inside the square root}
\Delta(x\cdot P,b,|P|)\geq C_0
\end{equation}
for all $|x|\leq 1+\epsilon$, $P\in D$ and $b$ satisfying \eqref{eq:lower and upper estimate of b to avoid internal reflection},
consequently, the formula \eqref{eq:definition of h(x,P,b)} can be extended and is then well defined for all these values.
In particular, \eqref{eq:definition of h(x,P,b)} can be differentiated with respect to $x$ for all $|x|\leq 1+\epsilon$
obtaining
\begin{equation}\label{eq:gradient of h in x}
\nabla_xh(x,P,b)
=
\dfrac{\kappa^{2}h(x,P,b)}{\sqrt{\Delta \(x\cdot P,b,|P|\)}}\,P\qquad \text{for $|x|\leq 1+\epsilon$.}
\end {equation}

{\bf Upper bound for $\nabla_xh(x,P,b)$.} 
From \eqref{eq:gradient of h in x}, we only need to estimate the extension $h$ from above and $\sqrt{\Delta}$ from below. 
To estimate $h$ we have from \eqref{eq:upper bound for Delta} that
\[
h(x,P,b)\leq C_\kappa\,|P|,\qquad \text{for $|x|\leq 1+\epsilon$ and $\kappa|P|<b<|P|$,}
\] 
and to estimate $\sqrt{\Delta}$ from below we use \eqref{eq:lower bound for quantity inside the square root}
obtaining
\begin{equation}\label{eq:upper estimate of the gradient of h}
|\nabla_xh\(x,P,b\)|
\leq 
\dfrac{C_k }{\sqrt{C_0}}\,|P|^2\qquad \text{for all $|x|\leq 1+\epsilon$}
\end{equation}
when $b$ satisfies  
\eqref{eq:lower and upper estimate of b to avoid internal reflection} and conditions \ref{eq:hypotheses 1} and \ref{eq:hypotheses 2} hold.

{\bf Bounds for $h_b$:}
\begin{align*}
h_b&=\dfrac{1-\dfrac12 \Delta^{-1/2} \Delta_b}{1-\kappa^2}=\dfrac{1- \Delta^{-1/2} \kappa^2 (b-x\cdot P)}{1-\kappa^2}
=
\dfrac{\Delta^{1/2}-  \kappa^2 (b-x\cdot P)}{\(1-\kappa^2\)\sqrt{\Delta}}
=
\dfrac{1}{1-\kappa^2}+\dfrac{\kappa^2\,\(x\cdot P-b\)}{\(1-\kappa^2\)\sqrt{\Delta}}.
\end{align*}
From \eqref{eq:definition of Deltat}, $\Delta(x\cdot P,b,|P|)> \kappa^2\(x\cdot P-b\)^2$ for $|x|\leq 1$.  
If $b$ satisfies \eqref{eq:lower and upper estimate of b to avoid internal reflection}, then from \eqref{eq:b less than or equal to x.P} $x\cdot P\geq b$ (that avoids internal reflection), 
and so $\sqrt{\Delta}\geq \kappa\,\(x\cdot P-b\)$. Therefore
\begin{equation}\label{eq:G_b is negative}
\dfrac{1}{1-\kappa^2}\leq  h_b\leq \dfrac{1}{1-\kappa^2}+\dfrac{\kappa}{1-\kappa^2}=\dfrac{1}{1-\kappa}.
\end{equation}

{\bf Bounds for $\mathcal G_{ij}$:}

For $i\neq j$, $\kappa|P_i|<a\leq (1+\kappa) \,r_0+\kappa |P_i|$,
$\kappa|P_j|<b\leq (1+\kappa) \,r_0+\kappa |P_j|$, 
and $|x|\leq 1+\epsilon$ let's define
\begin{equation}\label{eq:definition of Gij difference in x}
\mathcal G_{ij}\(x,a,b\):=h\(x,P_i,a\)-h\(x,P_j,b\).
\end{equation}
From the analysis above this function is differentiable with respect to $x$ for $|x|\leq 1+\epsilon$.
Suppose at some $|x|\leq 1+\epsilon$ and for some $a,b$ with $\kappa|P_i|+\delta_i\leq a\leq (1+\kappa) \,r_0+\kappa |P_i|; \kappa|P_j|+\delta_j\leq b\leq (1+\kappa) \,r_0+\kappa |P_j|$ we would have
\begin{equation}\label{eq:vanishing gradient in x}
\nabla_x\mathcal G_{ij}\(x,a,b\)=0.
\end{equation}
Then from \eqref{eq:gradient of h in x}
\[
\dfrac{\kappa^{2}h(x,P_i,a)}{\sqrt{\Delta \(x\cdot P_i,a,|P_i|\)}}\,P_i
=
\dfrac{\kappa^{2}h(x,P_j,b)}{\sqrt{\Delta \(x\cdot P_j,b,|P_j|\)}}\,P_j
\]
and since the coefficients in front of $P_i$ and $P_j$ are not zero it follows that $P_i$ is a multiple of $P_j$, violating the visibility Condition \ref{eq:hypotheses 2} taking $X=0\in Q_{r_0}$.

Therefore by continuity
\begin{equation}\label{eq:lower bound for mathcal Gij in x}
\min_{\substack{\\ \kappa|P_i|+\delta_i\leq a \leq (1+\kappa) \,r_0+\kappa |P_i|\\ \kappa|P_j|+\delta_j\leq b\leq (1+\kappa) \,r_0+\kappa |P_j|}}\left|\nabla_x\mathcal G_{ij}\(x,a,b\) \right|= \lambda>0
\end{equation}
for all $x\in \Omega$.

The estimate in the following proposition will be used in the proof of Theorem \ref{thm:one sided Lipschitz estimate} via Proposition \ref{HausdorffEstimate}. Its proof requires the structural condition \eqref{eq:structural condition ij}.
$x(u)$ denotes a point in the unit sphere $S^{n-1}$ with $u=(u_1,\cdots ,u_{n-1})$ being spherical coordinates.

\begin{proposition}\label{prop:structural condition for lower estimates of gradients}
Fix $i,j$ with $1\leq i,j\leq N$ and $i\neq j$,
$F\subset \R^{n-1}$ with $x(F)=\Omega$.

We assume the following structural condition:
if $\Pi_{ij}$ is the plane containing the origin $O$ and the points $P_i,P_j$,
then 
\begin{equation}\label{eq:structural condition ij}
\Omega\cap \Pi_{ij}=\emptyset.
\end{equation}
If $\delta_i,\delta_j$ are positive, then 
there exists a constant $\lambda>0$ depending on $\delta_i,\delta_j$ and $F$ such that
\begin{equation}\label{eq:lower bound for mathcal Gij}
\min_{\substack{\\ \kappa|P_i|+\delta_i\leq a \leq (1+\kappa) \,r_0+\kappa |P_i|\\ \kappa|P_j|+\delta_j\leq b\leq (1+\kappa) \,r_0+\kappa |P_j|}}
\left|\nabla_u\(\mathcal G_{ij}\(x(u),a,b\)\) \right|\geq \lambda
\end{equation}
for all $u\in F$.\footnote{Notice that the minimum in \eqref{eq:lower bound for mathcal Gij} over the full range
$\kappa|P_i|< a \leq |P_i|;\,\, \kappa|P_j|< b\leq |P_j|$ is zero. Because, if for example, $b\to \(\kappa|P_j|\)^+$, then $\Delta(x(u)\cdot P_j,b)\to \kappa^2\(|P_j|-\kappa\,x\cdot P_j\)^2>0$ since $\kappa<1$. On the other hand, $h(x(u),P_j,b)\to 0$ as $b\to \(\kappa|P_j|\)^+$. Therefore from \eqref{eq:gradient of h in x}, $\nabla_u h\(x(u),P_j,b\)\to 0$ as $b\to \(\kappa|P_j|\)^+$; and similarly $\nabla_u h\(x(u),P_i,a\)\to 0$ as $a\to \(\kappa|P_i|\)^+$.}
\end{proposition}

\begin{proof}
By contradiction. Suppose at some $u\in F$ and for some $a,b$ with $\kappa|P_i|+\delta_i\leq a\leq (1+\kappa) \,r_0+\kappa |P_i|; \kappa|P_j|+\delta_j\leq b\leq (1+\kappa) \,r_0+\kappa |P_j|$ we have
\begin{equation}\label{eq:vanishing gradient}
\nabla_u\(\mathcal G_{ij}\(x(u),a,b\)\)=0.
\end{equation}
That is, $\dfrac{\partial }{\partial u_k}\(\mathcal G_{ij}\(x(u),a,b\)\)=\nabla_x \mathcal G_{ij}\(x(u),a,b\)\cdot x_{u_k}=0$ for $1\leq k\leq n-1$.
Since $\left\{x_{u_1},\cdots ,x_{u_{n-1}},x\right\}$ is an orthogonal frame, we get that the vector $\nabla_x \mathcal G_{ij}\(x(u),a,b\)$ is parallel to $x$.
From \eqref{eq:gradient of h in x} we then get that 
\[
\dfrac{\kappa^{2}h(x(u),P_i,a)}{\sqrt{\Delta \(x(u)\cdot P_i,a,|P_i|\)}}\,P_i
-
\dfrac{\kappa^{2}h(x(u),P_j,b)}{\sqrt{\Delta \(x(u)\cdot P_j,b,|P_j|\)}}\,P_j\text{ is parallel to } x(u),
\]
which is a contradiction with \eqref{eq:structural condition ij}.
\end{proof}

\begin{remark}\rm 
Therefore, if the target points $P_1,\cdots ,P_N$ satisfy 
\begin{equation}\label{eq:cross product condition}
\Omega \cap \Pi_{ij}=\emptyset,\quad \forall i\neq j,
\end{equation} 
then \eqref{eq:lower bound for mathcal Gij} holds for all $i\neq j$.
To understand this condition, let $\nu_{ij}$ be a normal to $\Pi_{ij}$, that is, $\nu_{ij}$ is parallel to the vector $\overrightarrow{OP_i}\times \overrightarrow{OP_j}$.
Given $\Omega\subset S^2$ in the upper sphere, let $\Omega^\perp\subset S^2$ the orthogonal set of vectors 
\[
\Omega^\perp=\{y\in S^2:\text{there exists $x\in \Omega$ such that $y\cdot x=0$}\}.
\]
So \eqref{eq:lower bound for mathcal Gij} holds for all $i\neq j$ if the set of vectors $\nu_{ij}$ is contained in the complement of $\Omega^\perp$.

For example, if the points $P_1,\cdots ,P_N$ lie on a plane through the origin that does not intersect $\Omega$, and so that any pair $(P_i,P_j)$ is not aligned with the origin, then \eqref{eq:cross product condition} holds.

\end{remark}

\begin{remark}\rm
To illustrate \eqref{eq:cross product condition}, suppose the target $D$ is contained on the plane $z=a$. We can select points in $D$ in the following way so that 
\eqref{eq:cross product condition} holds. Let $P_1\in D$ so that the line $OP_1$ does not intersect $\Omega$ and consider $\mathcal C_1$ the collection of all planes containing the points $O$ and $P_1$ that intersect $\Omega$. Pick $P_2\in D$ with $P_2\notin \mathcal C_1$. Let 
$\mathcal C_2$ be the collection of all planes containing the points $O$ and $P_2$ that intersect $\Omega$.
Pick $P_3\in D$ such that $P_3\notin \mathcal C_1\cup \mathcal C_2$. Next let $\mathcal C_3$ be the collection of all planes containing the points $O$ and $P_3$ that intersect $\Omega$, and pick $P_4\in D$ with $P_4\notin \mathcal C_1\cup \mathcal C_2\cup \mathcal C_3$. Continuing in this way we choose points in $D$ so that \eqref{eq:cross product condition} holds. 
\end{remark}

\section{Lipschitz estimate of the refractor map}\label{sec:Lipchitz estimate of the refractor map}
In this section, we will prove the one sided Lipschitz estimate \ref{eq:one sided Lipschitz estimate of Hi}, for the refractor measure. This result is a crucial ingredient to prove that the algorithm converges to the desired result in finitely many steps when applied to near field refractor problem.

Let $\mathbf{b} = (b_1, \cdots, b_N)  \in \prod_{i=1}^N\(\kappa|P_i|,|P_i|\)$ and let $\mathbf{e}_i$ be the unit direction in the $i$-th coordinate. 
For $\kappa|P_i| < b_i-t \leq b_i$, define $\mathbf{b}^t = \mathbf{b} - t \,\mathbf{e}_i$. 
The domain $\Omega\subset S^{n-1}$ of incident directions is identified with $F\subset \R^{n-1}$ so that $x(F)=\Omega$ where $x=x(u)$ and $u$ spherical coordinates. 
Next we define the sets 
\begin{equation}\label{voronoi1}
V^{\mathbf{b}}_{i,j} := \left\{x \in \Omega : h(x,P_i, b_i) \leq h(x,P_j, b_j) \right\},
\end{equation}
\begin{equation}\label{voronoi1t}
V^{\mathbf{b}^t}_{i,j} := \left\{x \in \Omega : h(x,P_i, b_i-t) \leq h(x,P_j, b_j) \right\},
\end{equation} 
\begin{equation}\label{voronoi2}
V^{\mathbf{b}}_i := \bigcap_{j \neq i} V^{\mathbf{b}}_{i,j} = \left\{x \in \Omega : \rho_{\mathbf{b}}(x) = h(x,P_i, b_i) \right\},
\end{equation}
where $\rho_{\mathbf{b}}(x)=\min_{1\leq k\leq N}h(x,P_k, b_k)$.

Since $t > 0$, from \eqref{eq:G_b is negative} $h$ is increasing in the last variable so $V^\b_{i, j} \subset V^{\b^t}_{i, j}$ for all $j \neq i$. Hence $V^\b_{i} \subset V^{\b^t}_{i}$. 
Since in the arguments in this section the vector $\mathbf{b} $ will be fixed,  we adopt the short-hand
\begin{equation*}
V_{i,j} := V^\b_{i,j}, \quad V^t_{i,j} := V^{\b^t}_{i,j}, \quad V_i := V^\b_{i}, \quad V^t_i := V^{\b^t}_{i}.
\end{equation*}
With this notation, for the refractor measure map given in Definition \ref{def:definition of refractor measure}, we have 
\begin{equation*}
\mathcal M_{\S({\bf b}),f} (P_i) =\int_{\mathcal T_{\mathcal S({\bf b})}(P_i)}
f(x)\,d\sigma(x)=\int_{V_i} f(x)\,d\sigma(x).
\end{equation*}
If for brevity we denote $\mathcal M_{\S({\bf b}),f} (P_i)$ by $H_i(\bf{b})$ we have ,
\begin{equation}\label{eq:definition of refractor measure Hi}
H_i\({\bf b}\) = \int_{V_i} f(x)\,d\sigma(x).
\end{equation}

Our goal is to prove the following one-sided Lipschitz estimate for  $H_i$.

\begin{theorem}\label{thm:one sided Lipschitz estimate}
Assume that \ref{eq:hypotheses 1} and \ref{eq:hypotheses 2} in Section \ref{subsec:statement of the problem} hold, and the target points $P_1,\cdots ,P_N$ satisfy \eqref{eq:cross product condition}.
Let $\delta_1,\cdots ,\delta_N$ be positive numbers and $\mathbf{b}=(b_1,\cdots ,b_N)$ satisfying
\begin{equation}\label{eq:bj are in a box away from kPi}
\kappa\, |P_j|+\delta_j\leq b_j\leq (1+\kappa) \,r_0+\kappa \, |P_j|,\qquad j=1,\cdots ,N.
\end{equation}
Then for each $1\leq i\leq N$ we have 
\begin{equation}\label{eq:one sided Lipschitz estimate of Hi}
0 \leq H_i(\mathbf{b}^t) - H_i(\mathbf{b})
\leq C_0\,\|f\|_{L^\infty(\Omega)}\, t
\end{equation}
for all $t$ with $\kappa\,|P_i|< b_i-t\leq b_i\leq (1+\kappa) \,r_0+\kappa \,|P_i|$,  where $\mathbf{b}^t = \mathbf{b} - t \,\mathbf{e}_i$.
$C_0$ is a positive constant depending only on the bounds for the derivatives up to order two of the functions $h\(x(u),P_i,b_i\)$ over $u\in F$ $\(\Omega=x(F)\)$ and over $b_i$
satisfying \eqref{eq:lower and upper estimate of b to avoid internal reflection};  in addition $C_0$ depends also on $\delta_i$, $N$, $\kappa$, and the constants in 
\ref{eq:hypotheses 1} and \ref{eq:hypotheses 2}.

\end{theorem}

\begin{proof}
We have from Lemma \ref{monotonicityoftrace}
that \begin{align*}
0 \leq H_i(\mathbf{b}^t) - H_i(\mathbf{b}) 
& = \int\limits_{V^t_i \backslash V_i} f(x) \ d\sigma(x).
\end{align*}

\noindent Using \eqref{voronoi2} we obtain
\begin{align*}
V^t_i \backslash V_i
&=
\left\{\cap_{j\neq i} V_{i,j}^t\right\}\setminus \left\{\cap_{r\neq i} V_{i,r} \right\}=
\left\{ \cap_{j\neq i} V_{i,j}^t\right\}\cap 
\left\{\left( \cap_{r\neq i} V_{i,r}\right)^c \right\}\\
&=
\left\{ \cap_{j\neq i} V_{i,j}^t\right\}\cap 
\left\{ \cup_{r\neq i} V_{i,r}^c \right\}= 
 \cup_{r\neq i}\left\{V_{i,r}^c\cap \left\{ \cap_{j\neq i} V_{i,j}^t\right\} 
  \right\}\\
  &\subset  \cup_{r\neq i}\left\{V_{i,r}^c\cap V_{i,r}^t
  \right\}=  \bigcup_{r\neq i}\left(V_{i,r}^t\setminus V_{i,r} 
  \right) .
\end{align*}

\noindent Hence,
\begin{equation}\label{containmentinannulus}
V^t_i \setminus V_i \subset \bigcup_{j \neq i} \left(V^t_{i,j} \setminus V_{i,j} \right).
\end{equation}

\noindent It follows that
\begin{align}
0 \leq H_i(\mathbf{b}^t) - H_i(\mathbf{b}) & = \int\limits_{V^t_i \setminus V_i} f(x) \ d\sigma(x)\notag \\
& \leq \|f\|_{L^{\infty}(\Omega)}\, \sigma\left(\bigcup_{j \neq i} \left(V^t_{i,j} \setminus V_{i,j} \right) \right)\notag \\
& \leq \|f\|_{L^{\infty}(\Omega)} \sum\limits_{j \neq i} \sigma \left(V^t_{i,j} \setminus V_{i,j} \right),\label{eq:sumofmeasuresforjneqi}
\end{align}
where $\sigma$ denotes the area measure in the sphere $S^{n-1}$.
We proceed to estimate $\sigma\left(V^t_{i,j} \setminus V_{i,j} \right)$ for $j \neq i$. Notice that, by definition of $V_{i,j}$,
\begin{align*}
V^t_{i,j} \setminus V_{i,j} 
& = 
\left\{x \in \Omega : h(x,P_i,b_i) \geq h(x,P_j,b_j) \geq h(x,P_i,b_i - t) \right\} \\
& = \left\{x \in \Omega : 0 \leq  h(x,P_i,b_i) - h(x,P_j,b_j) \leq h(x,P_i,b_i) - h(x,P_i,b_i-t) \right\}.
\end{align*}
If $E$ is a subset of the upper part of the sphere $S^{n-1}$, and $x(u)\in E$ with $u$ spherical coordinates, then there is $E'\subset \R^{n-1}$ such that $E=x(E')$ and from the formula of change of variables
\[
\int_E f(x)\,d\sigma(x)=\int_{E'} f(x(u))\,\sqrt{\det Dx^T Dx }\,du,
\]
where $Dx=\dfrac{\partial x_k}{\partial u_j}$, $x=(x_1,\cdots ,x_n)$, $u=(u_1,\cdots ,u_{n-1})$.
Recall $\Omega=x(F)$ with $F\subset \R^{n-1}$.
Since $V^t_{i,j} \setminus V_{i,j}\subset \Omega\subset S^{n-1}$, let $F^t_{i,j}\subset \R^{n-1}$
\[
V^t_{i,j} \setminus V_{i,j}=x\(F^t_{i,j}\),
\]
and so the surface measure
\[
\sigma \(V^t_{i,j} \setminus V_{i,j}\)\leq c\,|F^t_{i,j}|_{n-1},
\]
where $|\cdot |_{n-1}$ denotes the $(n-1)$-dimensional Lebesgue measure and $c$ a constant.
If 
\begin{equation}\label{eq:inequality for bi-t}
\kappa|P_i|<b_i-t\leq (1+\kappa) \,r_0+\kappa |P_i|,
\end{equation}
then from the bound \eqref{eq:G_b is negative} for $h_b$ - only depending on $\kappa$ - and the mean value theorem we get 
\begin{align*}
h(x,P_i,b_i) - h(x,P_i,b_i-t) &= h_b(x,P_i, \xi_i) \cdot t \\
& \leq C(\kappa)\, t,
\end{align*}
for all $x\in \Omega$.
Therefore
\begin{equation}\label{generalizedannulussubsetH}
F^t_{i,j} \subset \left\{u \in F : 0 < \mathcal{G}_{ij}(x(u),b_i,b_j) \leq C(\kappa)\,t \right\},
\end{equation}
for all $t$ satisfying \eqref{eq:inequality for bi-t} and $j\neq i$, with $\mathcal{G}_{ij}$ from \eqref{eq:definition of Gij difference in x}.
 
The last set is a region contained between two level sets of the function $\mathcal{G}_{ij}$ and we now estimate the measure of this set. Let us first recall the co-area formula, \cite[Section 3.4.2, Theorem 1]{EG}.

\begin{proposition} Let $\psi: \mathbb{R}^n \rightarrow \mathbb{R}$ be Lipschitz, and $\Sigma \subset \mathbb{R}^n$ measurable. Then 
\begin{equation}\label{coareaformula}
\int\limits_{\Sigma} |D\psi(x)| \ dx = \int\limits_{-\infty}^{\infty} \mathcal{H}^{n-1} (\Sigma \cap \psi^{-1}(s)) \ ds,
\end{equation}

\noindent where $\mathcal{H}^{n-1}(\cdot)$ denotes $(n-1)$-dimensional Hausdorff measure.

\end{proposition}

This has the following simple corollary.

\begin{corollary} Let $\psi: \bar \Omega \rightarrow \mathbb{R}$ be Lipschitz, with $\inf_\Omega |D\psi| \geq \lambda > 0$, $-\infty \leq a \leq b \leq \infty$ and $\Omega\subset \R^n$ a bounded set. 
Then
\begin{equation}\label{CoAreaCorollary1}
\mathcal{L}^n(\left\{x \in \Omega : a \leq \psi(x) \leq b \right\}) \leq \frac{1}{\lambda} \int\limits_{a}^{b} \mathcal{H}^{n-1} \left(\Omega\cap\psi^{-1}(s)\right) \, ds,
\end{equation}
$\mathcal{L}^n$ being the $n$-dimensional Lebesgue measure.
\end{corollary}

From \eqref{eq:upper estimate of the gradient of h} and the lower bound \eqref{eq:lower bound for mathcal Gij} we can apply the corollary to the function $\psi(u)=\mathcal{G}_{ij}(x(u),b_i,b_j)$, $j\neq i$, to conclude that
\begin{equation}\label{CoAreaCorollary2}
\mathcal{L}^{n-1}\left(\left\{u \in F : 0 < \mathcal{G}_{ij}(x(u),b_i,b_j) \leq C\,t \right\} \right) \leq \frac{1}{\lambda} \int\limits_0^{C\,t} \mathcal{H}^{n-2} \left(F \cap \mathcal{G}_{ij}^{-1}(s)\right) \ ds;\quad C=C(\kappa).
\end{equation}

We now show that the integrand on the right hand side of \eqref{CoAreaCorollary2} is uniformly bounded for each $s$ in the range of $\mathcal{G}_{ij}$. For this, we need the following \cite[Prop. 5.5]{AG17}.

\begin{proposition}\label{HausdorffEstimate} Suppose $\Omega \subset \mathbb{R}^n$ is a smooth, bounded domain and $\psi \in C^2(\bar \Omega)$ satisfies $\min\limits_{x \in \overline{\Omega}} |D\psi(x)| \geq \lambda > 0$ and $\|D\psi\|_{L^{\infty}(\Omega)}, \|D^2 \psi\|_{L^{\infty}(\Omega)}$ are both finite. For any $s \in \mathrm{Range}(\psi)$, let $\Gamma_s = \left\{x \in \Omega : \psi(x) = s \right\}$. Then there exists a constant $K = K\left(\lambda, \|D\psi\|_{L^{\infty}(\Omega)}, \|D^2 \psi\|_{L^{\infty}(\Omega)}\right)$ such that
\begin{equation}\label{HausdorffMeasureEstimate}
\mathcal{H}^{n-1}(\Gamma_s) \leq \mathcal{H}^{n-1}(\partial \Omega) + K\, \mathcal{L}^n(\Omega).
\end{equation}
\end{proposition}
We can now complete the proof of the desired Lipschitz estimate. 
Assuming \eqref{eq:cross product condition}, then \eqref{eq:lower bound for mathcal Gij} holds and so we can apply Proposition \ref{HausdorffEstimate} when $n\rightsquigarrow n-1$, with $\Omega$ replaced by $F$, to the function $\psi(u) = \mathcal{G}_{ij}(x(u),b_i,b_j)$,
for $b_i$ and $b_j$ satisfying $\kappa|P_i|+\delta_i\leq b_i \leq (1+\kappa) \,r_0+\kappa |P_i|$, $\kappa|P_j|+\delta_j\leq b_j \leq (1+\kappa) \,r_0+\kappa |P_j|$,
 provided $\|D\psi\|_{L^{\infty}(F)}, \|D^2 \psi\|_{L^{\infty}(F)}$ are both finite.
That $\|D\psi\|_{L^{\infty}(F)}<\infty$ follows from \eqref{eq:bj are in a box away from kPi} and \eqref{eq:upper estimate of the gradient of h}, and that $\|D^2 \psi\|_{L^{\infty}(F)}<\infty$ follows computing $D^2\psi$ using \eqref{eq:gradient of h in x}, \eqref{eq:lower bound for quantity inside the square root}, and 
\eqref{eq:upper estimate of the gradient of h}. 
Therefore \eqref{CoAreaCorollary2} implies
\begin{equation*}
\mathcal{L}^{n-1} \left(\left\{u \in F : 0 < \mathcal{G}_{ij}(x(u),b_i,b_j) \leq C\,t \right\}\right) 
\leq \frac{C\,t}{\lambda} \(\mathcal{H}^{n-2}(\partial F) + K \mathcal{L}^{n-1}(F)\),\quad j\neq i.
\end{equation*}
Hence from \eqref{generalizedannulussubsetH}
\begin{align*}
\sigma \(V^t_{i,j} \setminus V_{i,j}\)\leq c\,|F^t_{i,j}|_{n-1}
\leq
\frac{C\,t}{\lambda} \,\(\mathcal{H}^{n-2}(\partial F) + K \,\mathcal{L}^{n-1}(F)\),\qquad j\neq i
\end{align*}
for each $t$ satisfying \eqref{eq:inequality for bi-t}.
Finally, adding these inequalities over $j\neq i$, from \eqref{eq:sumofmeasuresforjneqi} we then obtain the desired Lipschitz estimate \eqref{eq:one sided Lipschitz estimate of Hi}
with $C_0$ a constant depending on $F$, $N$, $\kappa$, $\delta_i$, and bounds for the derivatives up to order two of $h$. 
\end{proof}

\section{Admissible vectors for the iterative method}\label{subsec:geometric assumptions on Omega and D}

Recall we have distinct points $P_1,\cdots ,P_N$ in $D$, $g_i>0$, $i=1,\cdots ,N$ satisfying the conservation condition 
\eqref{energy conservation condition} 
where $f>0$ a.e. in $\Omega$. 
And also assume the configuration conditions \ref{eq:hypotheses 1} and \ref{eq:hypotheses 2}. 

In the following proposition we introduce the set of admissible vectors that will be used in the iterative method.
We remark that the Proposition gives that vectors in the admissible set $W_\delta$ have components bounded uniformly away from $\kappa|P|$.
\begin{proposition}\label{prop:lower estimate of bi}
Suppose $b_1$ satisfies
\begin{equation}\label{eq:condition on b1}
\kappa|P_1|<b_1\leq (1-\kappa) \,r_0+\kappa |P_1|, \footnote{Notice this implies that $b=b_1$ satisfies the weaker inequality \eqref{eq:lower and upper estimate of b to avoid internal reflection}. A reason to assume \eqref{eq:condition on b1} is to show the set $W_\delta$ is non empty.} 
\end{equation}
and let $0<\delta<g_1/(N-1)$. Consider the set\footnote{Notice that the bounds for $b_j$ imply from \eqref{eq:lower and upper estimate of b to avoid internal reflection} that the oval $h(x,P_j,b_j)$ refracts all $x\in \Omega$ into $P_j$.}
\begin{align*}
W_\delta=
&\left\{ \(b_2,\cdots ,b_N\):\kappa |P_i|<b_i<(1+\kappa) \,r_0+\kappa |P_i| \right.\\
&\qquad \left. \text{and $\int_{\mathcal T_{\S({\bf b})}(P_i)}f(x)\,dx\leq g_i+\delta$ for $2\leq i\leq N$ with
${\bf b}=(b_1,b_2,\cdots ,b_N)$} \right\}.
\end{align*}
Notice that from \ref{eq:hypotheses 1} and \ref{eq:hypotheses 2}, $(1+\kappa)r_0+\kappa |P_j|<|P_j|$.

Then $W_\delta\neq \emptyset$ and for $0<\alpha:=\dfrac{1-\kappa}{1+\kappa}\,\(b_1-\kappa |P_1|\)\(<(1+\kappa)\,r_0\)$ we have 
\begin{equation}\label{eq:lower estimate for  bi}
\text{if $(b_2,\cdots, b_N)\in W_\delta$ then $b_i\geq \kappa|P_i|+\alpha$ for $2\leq i\leq N$}.
\end{equation}
\end{proposition}
\begin{proof}
Let us first prove the second part of the proposition. 
Let $(b_2,\cdots, b_N)\in W_\delta$ and consider ${\bf b} = (b_1, b_2, \cdots, b_N).$
Since $\mathcal T_{\S({\bf b})}\(\cup_1^N P_i\)=\Omega$ and $\mathcal T_{\S({\bf b})}(P_i)\cap \mathcal T_{\S({\bf b})}(P_j)$ has surface measure zero for $i\neq j$, we have
\[
\int_\Omega f(x)\,dx
=
\sum_1^N \int_{\mathcal T_{\S({\bf b})}(P_i)}f(x)\,dx
\leq 
\int_{\mathcal T_{\S({\bf b})}(P_1)}f(x)\,dx+\sum_2^N \(g_i+\delta\)
\]
which from \eqref{energy conservation condition} implies that
\[
\int_{\mathcal T_{\S({\bf b})}(P_1)}f(x)\,dx\geq g_1-(N-1)\,\delta>0.
\]
Since $f>0$ a.e., we then get that the surface measure of the set $\mathcal T_{\S({\bf b})}(P_1)$ is positive. 
From \cite[Lemma 5.3]{GH14}, $\S({\bf b})$ is Lipschitz and so the set of singular points has measure zero. 
Hence there exists a point $x_0\in \mathcal T_{\S({\bf b})}(P_1)$ non singular for $\rho$. That is, there exists $\kappa|P_1|<\bar b<|P_1|$ such that the oval with radius  $h\(x,P_1,\bar b\)$ supports $\rho$ at $x_0$, that is, $\rho(x)\leq h\(x,P_1,\bar b\)$ for all $x\in \Omega$ with equality at $x=x_0$. On the other hand, by definition of $\rho$, $\rho(x)\leq h(x,P_1,b_1)$ and so $h\(x_0,P_1,\bar b\)=\rho(x_0)\leq h(x_0,P_1,b_1)$ implying $\bar b\leq b_1$. We claim that $\bar b=b_1$. If it were $\bar b<b_1$, then $\rho(x)\leq h\(x,P_1,\bar b\)< h\(x,P_1,b_1\)$ for all $x\in \Omega$. Hence $\rho(x)=\min_{2\leq i\leq N}h(x,P_i,b_i)$, and so at $x_0$ there would exist $h(\cdot ,P_i,b_i)$, for some $i\neq 1$, supporting $\rho$ at $x_0$.
That is, the ovals $h\(\cdot ,P_1,\bar b\)$ and $h(\cdot ,P_i,b_i)$ with $i\neq 1$ would support $\rho$ at $x_0$ and therefore $x_0$ would be a singular point, a contradiction. The claim is then proved.
Hence $h(x_0,P_1,b_1)=\rho(x_0)\leq h(x_0,P_i,b_i)$  for $i\neq 1$.
From the estimates for the ovals in Lemma \ref{lm:minandmaxofovalskappa<1} we have
\[
h(x_0,P_1,b_1)\geq \dfrac{b_1-\kappa|P_1|}{1+\kappa}
\quad \text{and} \quad 
h(x_0,P_i,b_i)\leq \dfrac{b_i-\kappa|P_i|}{1-\kappa},\quad i\neq 1,
\]
implying
\[
b_i\geq \kappa |P_i|+\dfrac{1-\kappa}{1+\kappa}\,\(b_1-\kappa |P_1|\)=\kappa|P_i|+\alpha,\quad i\neq 1,
\]
with $\alpha:=\dfrac{1-\kappa}{1+\kappa}\,\(b_1-\kappa |P_1|\)$. 
From \eqref{eq:condition on b1} and since $\dfrac{1-\kappa}{1+\kappa}<1$, 
$
\alpha \leq (1-\kappa)r_0(<(1+\kappa)r_0),
$
for $i\neq 1$, showing \eqref{eq:lower estimate for  bi}.

Finally, to show $W_\delta\neq \emptyset$, let $b_1$ satisfy \eqref{eq:condition on b1} and 
construct $(b_2,\cdots ,b_N)\in W_\delta$.
For this, it is enough to show the existence of $\kappa|P_j|<b_j<(1+\kappa)r_0+\kappa |P_j|$ such that $h(x,P_1,b_1)\leq h(x,P_j,b_j)$ for $2\leq j\leq N$ and $x\in \Omega$. Because with this choice we would have that $\mathcal T_{\S({\bf b})}(P_j)$ has surface measure zero for $j\geq 2$. 
We write $b_1-\kappa|P_1|=\sigma (1-\kappa)r_0$ for some $0<\sigma\leq1$, and let $b_j=\kappa|P_j|+\sigma (1+\kappa) r_0$, for $j=2,\cdots ,N$. 
Then, once again from the estimates for the ovals Lemma \ref{lm:minandmaxofovalskappa<1},
\begin{align*}
h(x,P_1,b_1)\leq \dfrac{b_1-\kappa|P_1|}{1-\kappa}=\sigma r_0
=
\dfrac{\sigma (1+\kappa) r_0}{1+\kappa}
=
\dfrac{b_j-\kappa|P_j|}{1+\kappa} \leq h(x,P_j,b_j),
\end{align*}
for $j\geq 2$ and we are done.
\end{proof}

\section{Abstract Algorithm}\label{sec:abstract algorithm}

We present here an algorithm, that in conjunction with the results previously obtained,  will be applied in Section \ref{sec:final solution of the refractor problem} to obtain a near field refractor satisfying \eqref{conditionfornumericalsolution}.  This type of algorithm has been used in \cite{B79}, \cite{CO08}, \cite{DGM17}, \cite{AG17}, and \cite{K14}.
Here the presentation is in an abstract setting so that it can be applied to solve other problems.

Let $G:\prod_{i=1}^N(\alpha_i,\beta_i)\to \R^N_{\geq 0}$ be a function, $G(\b)=\left(G_1(\b),\cdots ,G_N(\b) \right)$, $\b=(b_1,\cdots ,b_N)$, satisfying the following properties:
\begin{enumerate}
\item[(a)] $G$ is continuous on $\prod_{i=1}^N(\alpha_i,\beta_i)$;
\item[(b)] for each $1\leq i\leq N$, and $\alpha_i<s\leq t<\beta_i$
\[
G_i(b_1,\cdots,b_{i-1},t,b_{i+1},\cdots ,b_N)\leq G_i(b_1,\cdots,b_{i-1},s,b_{i+1},\cdots ,b_N),\text{ and}
\]
\[
G_j(b_1,\cdots,b_{i-1},t,b_{i+1},\cdots ,b_N)\geq G_j(b_1,\cdots,b_{i-1},s,b_{i+1},\cdots ,b_N)\qquad \forall j\neq i.
\]
\item[(c)] for each $1\leq i\leq N$ there is $C_i>0$ such that 
\[
\lim_{t\to \alpha_i^+}G_i(b_1,\cdots,b_{i-1},t,b_{i+1},\cdots ,b_N)=C_i
\]
for all $\b=(b_1,\cdots ,b_N)$.
\end{enumerate}

Let $f_1,\cdots ,f_N,\delta$ be positive numbers satisfying
\[
f_i-\delta>0\text{ and } C_i>f_i \text{ for }\quad i=2,\cdots ,N;
\]
and let us fix $b_1^0\in (\alpha_1,\beta_1)$ and define the set
\[
W=\left\{\b=(b_1^0,b_2,\cdots ,b_N)\in \prod_{i=1}^N(\alpha_i,\beta_i): G_i(\b)\leq f_i+\delta \text{ for $i=2,\cdots ,N$} \right\}.
\]
Our purpose is to present an iterative procedure to construct a vector ${\bf b}\in W$ so that
\begin{equation}\label{eq:desired approximation}
\left|G_i(\b)- f_i\right|<\delta \qquad \text{for $2\leq i\leq N$.}
\end{equation}
This will be done by successively decreasing the coordinates of the vectors involved.
In addition, we will show also that if the function $G$ satisfies a Lipschitz condition, then the procedure terminates in a finite number of iterations.

\subsection{Description of the algorithm}
Suppose $W\neq \emptyset$ and pick $\b^1\in W$. 
We will construct $N-1$ intermediate consecutive vectors $\b^2,\cdots ,\b^N$ associated with $\b^1$ in the following way.

{\bf Step 1.}
We first test if $\b^1=\(b_1^0,b_2,\cdots ,b_N\)$ satisfies the inequalities:
\begin{equation}\label{eq:testinequality}
f_2-\delta \leq G_2(\b^1)\leq f_2+\delta.
\end{equation}
Notice that the last inequality in \eqref{eq:testinequality} holds since $\b^1\in W$.
If $\b^1$ satisfies \eqref{eq:testinequality}, then we set $\b^2=\b^1$ and we proceed to Step 2 below.
If $\b^1$ does not satisfy \eqref{eq:testinequality}, then
\begin{equation}\label{eq:Glessthanf2-delta}
G_2(\b^1)<f_2-\delta. 
\end{equation}
We shall pick $b_2^*\in (\alpha_2,b_2)$, and leave all other components fixed, so that the new vector 
$\b^2=\(b_1^0,b_2^*,b_3,\cdots ,b_N\)\in W$, and satisfies
\begin{equation}\label{eq:Glessthanf2-deltastronger}
f_2\leq G_2(\b^2)\leq f_2+\delta.
\end{equation}
Let us see this is possible. From (b) above, and since $\b^1\in W$,
\[
G_j(b_1^0,t,b_3,\cdots,b_N)\leq G_j(b_1^0,b_2,b_3,\cdots ,b_N)\leq f_j+\delta \qquad \text{for $\alpha_2<t<b_2$ and $j\neq 2$.}
\]
From (c) above 
\[
\lim_{t\to \alpha_2^+}G_2(b_1^0,t,b_3,\cdots ,b_N)=C_2.
\]
From (a), $G_2(b_1^0,t,b_3,\cdots ,b_N)$ is continuous for $t\in (\alpha_2,b_2)$.
Since 
\[
C_2>f_2>f_2-\delta,
\] 
then by the intermediate value theorem there is $b_2^*\in (\alpha_2,b_2)$ such that 
\[
G_2(b_1^0,b_2^*,b_3,\cdots ,b_N)=f_2,
\]
and therefore \eqref{eq:Glessthanf2-deltastronger} holds and $\b^2\in W$.

Therefore, if the vector $\b^1$ does not satisfy \eqref{eq:testinequality}, we have then constructed a vector $\b^2\in W$ that satisfies \eqref{eq:Glessthanf2-deltastronger} which is stronger than \eqref{eq:testinequality}.
\newline
{\bf Step 2.}
Next we proceed to test the inequality 
\begin{equation}\label{eq:testinequality2}
f_3-\delta\leq G_3(\b^2)\leq f_3+\delta,
\end{equation}
with $\b^2$ the vector constructed in Step 1.
If $\b^2$ satisfies \eqref{eq:testinequality2}, we set $\b^3=\b^2$ and we proceed to the next step.
If $\b^2$ does not satisfy \eqref{eq:testinequality2}, then 
\[
G_3(\b^2)< f_3-\delta.
\]
From (a), (b), (c) above, we proceed as before, now decreasing the value of $b_3$, the third component of the vector $\b^2$, and since 
\[
C_3>f_3>f_3-\delta
\]
construct a vector $\b^3\in W$ such that 
\[
f_3\leq G_3(\b^3)\leq f_3+\delta,
\]
and in particular,
\eqref{eq:testinequality2} holds for $\b^3$.
Notice that because of condition (b) 
we cannot conclude that the newly constructed vector $\b^3$ satisfies \eqref{eq:testinequality}.
\newline

{\bf Step $N-1$.}
We proceed to test the inequality 
\begin{equation}\label{eq:testinequalityN-1}
f_N-\delta \leq G_N(\b^{N-1})\leq f_N+\delta,
\end{equation}
where $\b^{N-1}$ is the vector from Step $N-2$. If this holds we set $\b^N=\b^{N-1}$. Otherwise,  we have
\[
G_N(\b^{N-1})<f_N-\delta,
\]
and proceeding as before, by decreasing the $N$th-component of $\b^{N-1}$, we obtain a vector $\b^N\in W$
\[
f_N\leq G_N(\b^{N})\leq f_N+\delta,
\]
as long as 
\[
C_N>f_N>f_N-\delta.
\]
In this way, if 
\[
C_j>f_j>f_j-\delta\qquad j=2,\cdots ,N,
\]
starting from a fixed vector $\b^1\in W$, we have constructed intermediate vectors $\b^2,\cdots ,\b^N$ all belonging to $W$ and satisfying the inequalities:
\[
f_j-\delta\leq G_j(\b^j)\leq f_j+\delta\qquad j=2,\cdots ,N.
\]
Notice that if $\b^1 = \b^N$, then the vector $\b^1$ satisfies \ref{eq:desired approximation}. If not, we repeat the above steps starting with the last vector $\b^N.$

It is important to notice that by construction, the $\ell$-th components of $\b^{j-1}$ and $\b^j$ are all equal for $\ell\neq j$.
If for some $2\leq j\leq N$, $\b^{j-1}\neq \b^j$, then the $j$-th component of $\b^j$ is strictly less than the $j$-th component of $\b^{j-1}$. And so if we needed to decrease the $j$-th component of $\b^{j-1}$ to construct $\b^j$ it's because 
\[
G_j(\b^{j-1})<f_j-\delta,
\]
and then by construction $\b^j$ satisfies  
\[
f_j\leq G_j(\b^j)\leq f_j+\delta.
\]
Therefore combining the last two inequalities we obtain the following important inequality
\begin{equation}\label{eq:inequalityfordifferenceoftwobjs}
\delta< G_j(\b^j)-G_j(\b^{j-1}),\quad \text{for intermediate consecutive vectors $\b^j\neq \b^{j-1}$.}
\end{equation}

In summary, we started from a vector $\b^{1,1}\in W$ and constructed $N-1$ intermediate vectors $\b^{1,2},\cdots ,\b^{1,N}$ using the procedure described.
So we obtain in the first stage the finite sequence of vectors 
\[
\b^{1,1},\b^{1,2},\cdots , \b^{1,N}.
\]
For the second stage we repeat the construction now starting with the vector $\b^{1,N}$ and we get the finite sequence of vectors 
\[
\b^{2,1},\b^{2,2},\cdots , \b^{2,N}
\] 
with $\b^{2,1}=\b^{1,N}$.
For the third stage we repeat the process now starting with the last intermediate vector $\b^{2,N}$ obtained in the previous stage, obtaining the finite sequence of vectors 
\[
\b^{3,1},\b^{3,2},\cdots , \b^{3,N}
\]
with $\b^{3,1}=\b^{2,N}$.
Continuing in this way we obtain a sequence of vectors, in principle infinite,
\begin{equation}\label{eq:sequenceofvectorsintheiteration}
\b^{1,1},\cdots , \b^{1,N};\b^{2,1},\cdots ,\b^{2,N}; \b^{3,1},\cdots ,\b^{3,N};\cdots ;\b^{n,1},\cdots ,\b^{n,N};\b^{n+1,1},\cdots ,\b^{n+1,N};\cdots 
\end{equation}
with $\b^{2,1}=\b^{1,N},\b^{3,1}=\b^{2,N},\cdots ,\b^{n+1,1}=\b^{n,N},\cdots$.  
If for some $n$, the vectors in the $n$th-stage are all equal, i.e., 
$\b^{n,1}=\b^{n,2}=\cdots =\b^{n,N}:=\b^n$, then 
from the construction
\[
|G_j(\b^n)- f_j|\leq \delta,\quad \text{for $2\leq j\leq N$.} 
\]
Therefore, if we show that for some $n$ the intermediate vectors $\b^{n,1},\b^{n,2},\cdots,\b^{n,N}$ are all equal, we obtain the desired approximation \eqref{eq:desired approximation}.

Let us see what happens for $G_1$.
Suppose\footnote[1]{\label{footnote:application near field}For the application to the near field refractor $G_i({\mathbf b})=\int_{\mathcal T_{\mathcal S({\bf b})}(P_i)}
f(x)\,d\sigma(x)$ so \eqref{eq:conservationofenergyforvectorsinW} holds for each vector ${\mathbf b}$ because
\begin{align*}
&\sum_{i=1}^N G_i(\b)\\
&=\sum_{i=1}^N \int_{\mathcal T_{\mathcal S({\bf b})}(P_i)}
f(x)\,d\sigma(x)\\
&=\int_{\cup_{i=1}^N\mathcal T_{\mathcal S({\bf b})}(P_i)}
f(x)\,d\sigma(x)\quad \text{since $\mathcal T_{\mathcal S({\bf b})}(P_i)\cap \mathcal T_{\mathcal S({\bf b})}(P_j)$ has measure zero for 
$i\neq j$ and $f>0$ a.e.}\\
&=\int_\Omega f(x)\,d\sigma(x)=f_1+\cdots +f_N\quad \text{from the energy conservation assumption.}
\end{align*} }
\begin{equation}\label{eq:conservationofenergyforvectorsinW}
\sum_{i=1}^N G_i(\b^n) = \sum_{i=1}^N f_i.
\end{equation}
Then
\begin{align*}
|f_1-G_1(\b^n)|
&= \left|\sum_{j=2}^NG_j(\b^n) - f_j\right| 
\leq \sum_{j=2}^N |G_j(\b^n) - f_j| \leq N\,\delta.
\end{align*}
Therefore the vector $\b^n$ satisfies 
\[
|G_j(\b^n)- f_j|\leq \delta,\quad \text{for $2\leq j\leq N$, and $|G_1(\b^n)- f_1|<N\delta$}. 
\] 

Summarizing, if $G_i$ satisfy (a), (b), (c), \eqref{eq:conservationofenergyforvectorsinW} and
\[
C_j>f_j>f_j-\delta\qquad j=2,\cdots ,N,
\]
choosing $\delta = \epsilon/N$, we then obtain
\[
|G_j(\b^n)- f_j|\leq \epsilon,\quad \text{for $1\leq j\leq N$.} 
\]

\subsection{Convergence of the algorithm}
We will show here that the procedure described will always give, in an infinite number of steps, a vector ${\bf b}\in W$ satisfying \eqref{eq:desired approximation} 
provided the following holds:\footnote{We remark that for the application to the near field refractor $\alpha=\dfrac{1-\kappa}{1+\kappa}\,\(b_1^0-\kappa |P_1|\)$, see Proposition \ref{prop:lower estimate of bi}.} \begin{equation}\label{eq:lowerboundforall vectors in W}
 \text{there exists $\alpha>0$ such that if $\(b_1^0,b_2,\cdots ,b_N\)\in W$, then $b_j\geq \alpha_j+\alpha$ for all $2\leq j\leq N$.}
 \end{equation}

As pointed out in \eqref{eq:sequenceofvectorsintheiteration}, by using the procedure described above we obtain a sequence of vectors
\[
\b^{n,\ell} = \(b_1^0,b_2^{n,\ell},b_3^{n,\ell},\cdots ,b_N^{n,\ell}\)
\]
$n \in \mathbb N$ and $1 \leq \ell \leq N$ which can be listed as
\begin{equation*}
\b^{1,1},\cdots , \b^{1,N};\,\b^{2,1},\cdots ,\b^{2,N};\, \b^{3,1},\cdots ,\b^{3,N};\cdots ;\,\b^{n,1},\cdots ,\b^{n,N};\,\b^{n+1,1},\cdots ,\b^{n+1,N};\cdots 
\end{equation*}
Notice that in this listing, for a fixed $j,$ $2 \leq j \leq N$ the sequence $\{b^{n,\ell}_{j}\}$ of the $j^{th}$ entries is non-increasing;  that is,
\[
b_{j}^{n,\ell}\geq b_{j}^{m,k},\,\text{ for $n \leq m$ \, or for $n=m$ and $\ell \leq k$}
\]
and for $j =1,$ we have $b_{j}^{n,\ell}= b_1^0$ for all $n$ and for all $\ell.$
Moreover, since the vectors belong to $W,$ by assumption \eqref{eq:lowerboundforall vectors in W}, each  $j^{th}$ entry is bigger than or equal to $\alpha_j+\alpha$ for $2\leq j\leq N$. Therefore for any $j$ the limit of the $j^{th}$ entries exists and the limit is strictly bigger than $\alpha_j$.

Let $b^\infty_j$ be the limit of the $j^{th}$ entries, $j\geq 2$. 
Then the vector 
\[
\b^\infty=\(b_1^0,b^\infty_2,b^\infty_3,\cdots , b^\infty_N\)\in \{b_1^0\}\times \prod_{i=2}^\infty (\alpha_i,\beta_i)
\]
satisfies 
\begin{equation}\label{eq:limitinginequalityforbinfty}
f_j-\delta \leq G_j(\b^\infty )\leq f_j+\delta, \,j=2,\cdots ,N.
\end{equation}
In fact, fix $2\leq j\leq N$, the vector $\b^\infty$ is the limit of the vectors 
$\b^{i,j}$ as $i\to \infty$. But the vectors $\b^{i,j}$ verify 
\[
f_j-\delta \leq G_j(\b^{i,j})\leq f_j+\delta, \text{ for $i=1,2,\cdots $}.
\]
From assumption (a), $G_j$ is continuous for each $j$, taking the limit as $i\to \infty$ we obtain \eqref{eq:limitinginequalityforbinfty}. 

Assuming \eqref{eq:conservationofenergyforvectorsinW} for all vectors $\b\in W$, we conclude that \eqref{eq:limitinginequalityforbinfty} holds with $j=1$ and with $\delta$ replaced by $N\delta$.

\begin{remark}Notice that the argument above always gives a solution $\(b_1^0,b_2,\cdots ,b_N\)$ satisfying \eqref{eq:limitinginequalityforbinfty}.
To handle the case when $j=1$ we need an extra condition.
In fact, for \eqref{eq:limitinginequalityforbinfty}  
to hold for $j=1$, the conservation of energy condition \eqref{eq:conservationofenergyforvectorsinW} is sufficient.

Also notice that if the conservation of energy condition \eqref{eq:conservationofenergyforvectorsinW} is assumed, then the second condition in (b) implies the first condition in (b).
This is all applicable to the near field refractor in view of the Footnote \ref{footnote:application near field} before \eqref{eq:conservationofenergyforvectorsinW}.
\end{remark}

\subsection{If $G$ satisfies a Lipschitz estimate then the algorithm terminates in a finite number of steps}\label{subsec:algorithm converges in a finite number of steps}

Suppose that given $\delta_1,\cdots ,\delta_N$ positive there is a constant $M>0$ such that
\begin{align}\label{eq:estimateofGiisimpler}
G_i(\b-t\,\e_i) - G_i(\b)   
&\leq M\,
\,t,
\end{align}
for $\alpha_i< b_i-t\leq b_i\leq \beta_i$ and for each $\b\in \prod_{i=1}^N [\alpha_i+\delta_i,\beta_i]$ and for all $1\leq i\leq N$. 
Notice that from assumption (b) above, $G_i(\b-t\,\e_i) - G_i(\b)\geq 0$. 

We shall prove that the estimate \eqref{eq:estimateofGiisimpler} together with the assumption 
that $W$ satisfies \eqref{eq:lowerboundforall vectors in W},
implies that there is $n$ such that the vectors in the $n$th group $\b^{n,1},\b^{n,2},\cdots,\b^{n,N}$ are all equal, and we will also show an upper bound for the number of iterations.

Suppose we originate the iteration at $\b^0=\(b_1^0,b_2^0,\cdots ,b_N^0\)  \in W.$ 
Since by construction the coordinates of the vectors in the sequence \eqref{eq:sequenceofvectorsintheiteration} are decreased or kept constant, the $j$th coordinate of any vector in the sequence is less than or equal to $b_j^0$, $1\leq j\leq N$.  
In addition, from \eqref{eq:lowerboundforall vectors in W}, points in $W$ have first coordinate $b_1^0$ and their coordinates  bounded  below by $\alpha_j+\alpha$ for $j\geq 2$. Therefore, all terms in the sequence \eqref{eq:sequenceofvectorsintheiteration} are contained in the compact box $K=\{b_1^0\}\times \prod_{j=2}^N [\alpha_j+\alpha,b_j^0]$.
We want to show that there is $n_0$ such that the intermediate vectors $\b^{n_0,1},\b^{n_0,2},\cdots ,\b^{n_0,N}$ are all equal.
Otherwise, for each $n$ the intermediate vectors $\b^{n,1},\b^{n,2},\cdots ,\b^{n,N}$ are not all equal.
This implies that for each $n$ there are two consecutive intermediate vectors $\(b_1^0,b_2,b_3,\cdots,b_N\)$ and 
$\(b_1^0,\bar b_2,\bar b_3,\cdots,\bar b_N\)$, that are different. By construction of intermediate vectors, they can only differ in one coordinate, say that $b_j>\bar b_j$. Notice that $j$ depends on $n$, but there is $j$ and a subsequence $n_\ell$ such that there are two consecutive  intermediate vectors $\(b_1^0,b_2^{n_\ell},b_3^{n_\ell},\cdots,b_N^{n_\ell}\)$ and 
$\(b_1^0,\bar b_2^{n_\ell},\bar b_3^{n_\ell},\cdots,\bar b_N^{n_\ell}\)$ in each group 
$\b^{n_\ell,1},\cdots ,\b^{n_\ell,N}$ such that their $j$-th coordinates satisfy 
$b_j^{n_\ell}>\bar b_j^{n_\ell}$, and all other coordinates are equal. Also notice that since the coordinates are chosen in a decreasing manner we have $b_j^{n_\ell}>\bar b_j^{n_\ell}\geq b_j^{n_{\ell+1}}>\bar b_j^{n_{\ell+1}}$ for $\ell=1,\cdots $.
From \eqref{eq:inequalityfordifferenceoftwobjs} we then get
\begin{equation}\label{eq:Lipschitzbounddelta}
\delta
<G_j\left(b_1^0,\bar b_2^{n_\ell},\bar b_3^{n_\ell},\cdots,\bar b_N^{n_\ell}\right)-G_j\left(b_1^0,b_2^{n_\ell},b_3^{n_\ell},\cdots,b_N^{n_\ell}\right)=(*)
\end{equation}
for each $\ell\geq 1$. 
We write
\[
\(b_1^0,\bar b_2^{n_\ell},\bar b_3^{n_\ell},\cdots,\bar b_j^{n_\ell},\cdots \bar b_N^{n_\ell}\)
=
\(b_1^0,\bar b_2^{n_\ell},\bar b_3^{n_\ell},\cdots,b_j^{n_\ell}+\bar b_j^{n_\ell}-b_j^{n_\ell},\cdots \bar b_N^{n_\ell}\),
\]
and let $t:=\bar b_j^{n_\ell}-b_j^{n_\ell}<0$.
Since the vectors belong to $W$, from \eqref{eq:lowerboundforall vectors in W} $\bar b_j^{n_\ell}\geq \alpha_j+\alpha$ for $2\leq j\leq N$. Then from \eqref{eq:estimateofGiisimpler}
we obtain
\begin{equation}\label{eq:boundswithL}
(*)\leq -\left(\bar b_j^{n_\ell}-b_j^{n_\ell}\right)\,
M=M\,(b_j^{n_\ell}-\bar b_j^{n_\ell}),\qquad \forall \ell.
\end{equation}
On the other hand,
\begin{equation}\label{eq:estimateofsumofdifferencesofcoordinates}
\sum_{\ell=1}^\infty ( b_j^{n_\ell}-\bar b_j^{n_\ell})\leq b_j^0-(\alpha_j+\alpha),
\end{equation}
which contradicts \eqref{eq:Lipschitzbounddelta}. Therefore, the intermediate vectors $\b^{n_0,1},\b^{n_0,2},\cdots ,\b^{n_0,N}$ are all equal for some $n_0$.

Let us now estimate the number of iterations used.
Consider the groups of vectors \eqref{eq:sequenceofvectorsintheiteration} 
in the construction. We have proved the process stops at some $n_0$, i.e., all vectors in this group are equal.
Let us estimate $n_0$. 
Fix a coordinate $2\leq j\leq N$. Notice that in each group $k$, the $j^{th}$ coordinate of any vector in the group can decrease {\it at most only once} and only when passing from a vector $\b^{k,j-1}$ to the vector $\b^{k,j}$. Here $k$ denotes the group and $j$ the location in the group.
The $j^{th}$ coordinate of all vectors are at most $b_j^0$, the $j^{th}$ coordinate of the initial vector, and since the vectors belong to $W$ and so \eqref{eq:lowerboundforall vectors in W} holds, the $j^{th}$ coordinates are at least $\alpha_j+\alpha$.
So the change in the $j^{th}$ coordinate of a vector in the group one to the group $n_0$, is at most $b_j^0-\(\alpha_j+\alpha\)$.
On the other hand, on each group if the $j^{th}$ coordinate is decreased, from \eqref{eq:boundswithL},  it is decreased by at least $\dfrac{\delta}{M}$.
Having $n_0$ groups, the total decrease of the $j^{th}$ coordinate in passing from group one to group $n_0$ is at least 
\[
n_0\,\dfrac{\delta}{M},
\] 
which is in turn smaller than the total possible decrease, that is, we have
\[
n_0\,\dfrac{\delta}{M}\leq b_j^0-\(\alpha_j+\alpha\).
\]
Since this must hold for all the coordinates $2\leq j\leq N$ we obtain the bound
\[
n_0\leq \dfrac{M}{\delta} \max_{2\leq j\leq N}\left(b_j^0-\(\alpha_j+\alpha\) \right).
\]

\section{Application of the algorithm to the near field refractor}\label{sec:final solution of the refractor problem}
We set $G_i({\bf b})=H_i({\bf b})$ with $H_i$ given in \eqref{eq:definition of refractor measure Hi},
$\alpha_i=\kappa\,|P_i|$, and $\beta_i=(1+\kappa)\,r_0+\kappa\,|P_i|$, $1\leq i \leq N$. 
The continuity property (a) for $H_i$  is contained in the proof of Step 2 in \cite[Theorem 2.5]{GH14}. 
Properties (b) and (c) follow from Lemmas \ref{monotonicityoftrace} and \ref{forlarge_b_energyis0},
with $C_i=\int_\Omega f\,dx$.
The set $W=W_\delta$ in Proposition \ref{prop:lower estimate of bi}, and so assumption \eqref{eq:lowerboundforall vectors in W} is \eqref{eq:lower estimate for  bi} for the near field refractor problem. Finally, the Lipschitz estimate \eqref{eq:estimateofGiisimpler} follows applying Theorem \ref{thm:one sided Lipschitz estimate} with $\delta_1=b_1^0-\kappa\,|P_1|$ and 
$\delta_j=\alpha=\dfrac{1-\kappa}{1+\kappa}\,\(b_1^0-\kappa |P_1|\)$ for $j=2,\cdots ,N$. 
We then have everything in place to be able to apply the abstract algorithm to the near field refractor problem and we can obtain  the poly-oval refractor $\S(\b)$ that satisfies \ref{conditionfornumericalsolution}.

\section*{Acknowledgments}
It is a pleasure to thank Farhan Abedin for a careful reading of this paper and for very useful suggestions.


\begin{thebibliography}{9}
\bibitem[AG17]{AG17}  F. Abedin and C. E. Guti\'errez,  {\em An iterative method for generated Jacobian equations,} Calc. Var. PDEs, 56(101), 1-14, (2017).
\bibitem[B79]{B79} D. I. Bertsekas, {\em A distributed algorithm for the assignment problem}, Laboratory for Information and Decision Sciences Working Paper, MIT, Cambridge, Mass., March 1979.
\bibitem[BW80]{BW80} M. Born, and E. Wolf, {\em Principles of Optics,} Sixth ed., Pergamon Press, 1980.
\bibitem[CKO99]{CKO99} L.A. Caffarelli, S.A. Kochengin, and V. Oliker,  {\em On the numerical solution of the problem of reflector design with given far-field
scattering data}, Contemp. Math., 226, 13--32, 1999.
\bibitem[DGM17]{DGM17} R. De Leo, C. E. Guti\'{e}rrez\ and\ H. Mawi, {\em On the numerical solution of the far field refractor problem}, Nonlinear Anal. {\bf 157} (2017), 123--145. 
\bibitem[EvG92]{EG} L. C. Evans and R. F. Gariepy,   {\em Measure Theory and Fine Properties of Functions,} CRC Press, Boca Raton, FL, 1992.
\bibitem[Gu01]{Gu01} C. E. Guti\'errez, {\em The Monge-Amp\`ere Equation,} Birkha\"user, 2nd ed., 2016.
\bibitem[GH09]{GH09} C. E. Guti\'errez  and  Q. Huang,  {\em The refractor problem in reshaping light beams,} Arch. Rational Mech. Anal, 193(2009).
\bibitem[GH14]{GH14}   C. E. Guti\'errez  and  Q. Huang, {\em The near field refractor} , Annales de L'Institut Henri Poincar\'e (C) Analyse Non Lin\'eaire, vol. 31 (4), pp. 655-684, 2014.
\bibitem[GM10]{GM10} C. E. Guti\'{e}rrez\ and\ H. Mawi, {\em The refractor problem with loss of energy}, Nonlinear Anal. {\bf 82} (2013), 12--46.
\bibitem[K14]{K14} J. Kitagawa.  {\em An iterative scheme for solving the optimal
transportation problem,}  Calc. Var. PDEs (2014) 51:243--263.
\bibitem[KK65]{KK} M. Kline  and  I. Kay, {\em Electromagnetic theory and Geometrical Optics}, 1965, J Wiley and Sons, Inc.
\bibitem[MTW05]{MTW05} X.-N. Ma, N. Trudinger, and X.-J. Wang, {\em Regularity of potential functions of the optimal transportation problem}, Arch. Rational Mech. Anal. 177(2), 151-183, 2005.
\end{thebibliography}
\end{document}